\documentclass{article}
\usepackage[english]{babel}

\usepackage{amsmath,amsfonts,amssymb, amsthm, graphicx, pstricks, pst-text, geometry}

\theoremstyle{plain}

    \newtheorem{Corollary}{Corollary}[section]
    \newtheorem{Proposition}{Proposition}[section]
    \newtheorem{Conjecture}{Conjecture}[section]
\theoremstyle{remark}
    \newtheorem{Remark}{Remark}[section]
    \newtheorem{Example}{Example}[section]
\theoremstyle{definition}

\newcommand{\g}{\mathfrak{g}}
\newcommand{\h}{\mathfrak{h}}
\newcommand{\aff}{\mathfrak{aff}}

\newcommand{\ddx}[1]{ \dfrac{\partial}{\partial #1} }

\newcommand{\gl}{\mathfrak{gl}}
\newcommand{\sL}{\mathfrak{sl}}
\newcommand{\Complex}{\mathbb{C}}
\newcommand{\eps}{\varepsilon}

\newcommand{\ad}{\mathrm{ad}}

\newcounter{picture}
 \newcommand{\pic}[1]{\refstepcounter{picture}\label{#1}{Figure \arabic{picture}}\par}              

\title{The derived algebra of a stabilizer, families of coadjoint orbits, and sheets}                                     
\author{Anton Izosimov\footnote{Moscow State University and Higher School of Economics. E-mail: a.m.izosimov@gmail.com}}                 
\date{}

\begin{document}


\maketitle

\begin{abstract}
Let $\g$ be a finite-dimensional real or complex Lie algebra, and let $\mu \in \g^{*}$. In the first part of the paper, the relation is discussed between the derived algebra of the stabilizer of $\mu$ and the set of coadjoint orbits which have the same dimension as the orbit of $\mu$. In the second part, semisimple Lie algebras are considered, and the relation is discussed between the derived algebra of a centralizer and sheets. 
\end{abstract}

\section{Introduction}
Let $\g$ be a finite-dimensional real or complex Lie algebra. The group $G$ acts on the dual space $\g^{*}$ via the coadjoint action, and $\g^{*}$ is foliated into the orbits of this action. Consider the union of all orbits which have the same codimension $k$. Denote this union by $\g_{k}^{*}$. Each of the sets $\g^*_k$ is a quasi-affine algebraic variety.
The study of the varieties $\g_{k}^{*}$ was initiated by A.Kirillov in connection with the orbit method \cite{Kirillov}, which relates the unitary dual of $G$ to the set of coadjoint orbits $\g^{*}/G$. The sets  $\g_{k}^{*} / G$ appear in this picture as natural strata of $\g^{*}/G$, therefore it is important to understand the geometry of $\g_{k}^{*}$ for each $k$.\par \smallskip
So, consider a finite-dimensional real or complex Lie algebra $\g$. 
Let $\g_{\mu} = \{ x \in \g \mid \ad^{*}x(\mu) = 0\}$ be the stabilizer of an element $\mu \in \g^{*}$ with respect to the coadjoint representation of $\g$. In the present note, the following simple geometric fact is proved: {any element $\xi \in \g^*$ which is tangent to the variety $\g_{k}^{*}$ at a point $\mu\in \g^*$ vanishes on the derived algebra of $\g_{\mu}$}. As a corollary, {the codimenson of the set $\{\mu \in \g^* \mid \dim\,[\g_\mu, \g_\mu] \geq k \}$ is at least $k$}, which generalizes the well-known fact that the stabilizer of a generic element $\mu \in \g^*$ is Abelian.\par

In the second part of the note, a semisimple Lie algebra $\g$ is considered. In this case, the set $\g_{k}^{*}$ can be identified with the variety of adjoint orbits of codimension $k$. The irreducible components of this latter variety are called sheets. Let $a \in \g$, and let $\g^{a} = \{ x \in \g \mid [x,a] = 0\}$ be the centralizer of $a$. 
Then the above-formulated statement becomes the following:{ the derived algebra of the centralizer of $a$ is orthogonal to any sheet passing through $a$.} It was conjectured in the earlier version of the present paper \cite{prev} that {if $\g$ is a classical simple Lie algebra, and there is a unique sheet $S$ passing through $a \in \g$, then $[\g^a, \g^a]$ is exactly the orthogonal complement to $S$.} 
As it has been recently shown by A.Premet and L.Topley \cite{Premet}, this conjecture is true for any algebraically closed ground field of characteristic zero. The second conjecture, stating that if $S_1, \dots, S_k$ are sheets passing through $a$, then $[\g^a, \g^a]$ is the orthogonal complement to $\sum \mathrm{T}_aS_i$, remains open.
 
\section{The derived algebra of a stabilizer and families of coadjoint orbits}
\paragraph{Families of coadjoint orbits of the same dimension}
Let $\g$ be a real or complex Lie algebra. Let $\g_{\mu} = \{ x \in \g \mid \ad^{*}{x}(\mu) = 0\}$ be the stabilizer of an element $\mu \in \g^{*}$ with respect to the coadjoint action. 
Let $\g_{k}^{*} = \{\mu \in \g^{*} \mid \dim \g_{\mu} = k\}$. It is clear that $\g_{k}^{*}$ is a quasi-affine algebraic variety for each $k$, and $\g^{*}$ is a disjoint union of all $\g_{k}^{*}$.
The variety $\g_{k}^{*}$ can be also defined as the union of all coadjoint orbits of codimension $k$.\par
For any $\mu \in \g^*$, denote 
$\g^{*}(\mu) = \g_{\,\dim \g_{\mu}}^{*}$; in other words, $\g^{*}(\mu)$ is $\g_{k}^{*}$ passing through $\mu$.

\paragraph{Main statement}
\begin{Proposition}
	Let $\gamma$ be a smooth curve in $\g^*$ such that $\gamma(0) = \mu$ and $\gamma(t) \in \g^{*}(\mu)$ for all $t$. Then the tangent vector $\dot \gamma(0)$ vanishes on the derived algebra of $\g_{\mu}$:
	$$\langle \dot \gamma(0), [\g_{\mu}, \g_{\mu}] \rangle = 0.$$
\end{Proposition}
\begin{proof}
 Since $\dim \g_{\gamma(t)} = \dim g_{\mu}$ for all $t$, the bundle of stabilizers is locally trivial over the curve $\gamma$, and it is possible to choose a basis $e_{1}(t) , \dots, e_{k}(t)$ in $\g_{\gamma(t)}$ such that $e_i(t)$ depends smoothly on $t$. Since $e_{i}(t) \in \g_{\gamma(t)}$, the following equality holds:
 $$
 	\langle \gamma(t), [e_{i}(t), e_{j}(t)]\rangle = 0.
 $$
 Differentiating with respect to $t$ at $t = 0$, obtain
   $$
 	\langle \dot \gamma(0), [e_{i}(0), e_{j}(0)]\rangle + \langle \mu, [\dot e_{i}(0), e_{j}(0)]\rangle + \langle \mu, [e_{i}(0), \dot e_{j}(0)]\rangle= 0.
 $$
Since $e_{i}(t)$ are elements of the stabilizer, the last two terms vanish, and
  $$
 	\langle \dot \gamma(0), [e_{i}(0), e_{j}(0)]\rangle = 0,
 $$
 which implies that $\dot \gamma(0)$ vanishes on the derived algebra of $\g_{\mu}$, q.e.d.
\end{proof}
\begin{Remark}
	The proposition remains true if $\gamma(t)$ is only defined for $t \geq 0$ and the right derivative $\dot \gamma_{+}(0)$ exists. This may happen if $\g^{*}(\mu)$ has a singularity at $\mu$.
\end{Remark}
\begin{Corollary}\label{fundam}
	Consider the case when $\g^{*}(\mu)$ is smooth at the point $\mu$. Then
	\begin{enumerate}
		\item Each element of the tangent space $\mathrm{T}_{\mu}\g^{*}(\mu)$ vanishes on the derived algebra of $\g_{\mu}$: $$\langle \mathrm{T}_{\mu}\g^{*}(\mu), [\g_{\mu}, \g_{\mu}] \rangle = 0.$$
	\item The following inequality is satisfied:
	\begin{align}\label{ineq1}
		\dim\, [\g_{\mu}, \g_{\mu}] \leq \mathrm{codim}_{\mu}\, \g^{*}(\mu).
	\end{align}
	\item The equality 
	$$\dim\, [\g_{\mu}, \g_{\mu}] = \mathrm{codim}_{\mu} \,\g^{*}(\mu)$$  is satisfied if and only if $\mathrm{T}_{\mu}\g^{*}(\mu)$ is exactly the annihilator of $ [\g_{\mu}, \g_{\mu}]$.
	\end{enumerate}
\end{Corollary}
\begin{Remark}
	Inequality (\ref{ineq1}) shows that the derived algebra of a stabilizer cannot be too big. It resembles the following inequality for the index of a stabilizer: $\mathrm{ind}\, \g_\mu \geq \mathrm{ind}\, \g$ (Vinberg, see \cite{Vinberg}).
\end{Remark}
\begin{Corollary}\label{cor}
	The codimension of the set of elements $\mu \in \g^{*}$ such that $$\dim \,[\g_{\mu}, \g_{\mu}] \geq k$$ is at least $k$.
\end{Corollary}
\begin{Example}
	For regular $\mu$, obtain a well-known fact: $\g_{\mu}$ is abelian. Corollary \ref{cor} can be viewed as a natural generalization of this fact. It says that for a ``not too singular'' $\mu$, its stabilizer is almost Abelian.
\end{Example}
\begin{Example}
 	Let $\g$ be complex semisimple. Then, for a generic singular element $\mu$, the dimension of $[\g_{\mu}, \g_{\mu}]$ equals three. So, by Corollary \ref{cor}, the codimension of the singular set for a complex semisimple Lie algebra is at least three. On the other hand, it is well known that this codimension is exactly three.
\end{Example}

\begin{Example}
Suppose that the set of singular elements in $\g^*$ is a hypersurface. Then the stabilizer of a generic singular element is one of the following types:
	\begin{enumerate}
		\item Abelian;
		\item $\aff(1) \oplus \mbox{Abelian}$, where $\aff(1)$ is the Lie algebra of affine transformations of the line;
		\item $\h_{2n+1} \oplus \mbox{Abelian}$, where $\h_{2n+1}$ is the $2n+1$-dimensional Heisenberg algebra.
	\end{enumerate} 
	The class of Lie algebras for which the set of singular elements in $\g^*$ is a hypersurface is particularly important for integrable systems \cite{Bolsinov, JK}. 
\end{Example}
\begin{Remark}
Rewrite (\ref{ineq1}) as
\begin{align}\label{ineq2}
	\dim_{\mu} \g^{*}(\mu) - \dim O(\mu) \leq \dim \g_{\mu} - \dim\, [\g_{\mu}, \g_{\mu}]
\end{align}
where $O(\mu)$ is the coadjoint orbit of $\mu$. Coadjoint orbits form families and the difference $$\dim_{\mu} \g^{*}(\mu) - \dim O(\mu)$$ is exactly the local dimension of such a family. Inequality (\ref{ineq2}) estimates this dimension.
\end{Remark}
\begin{Example}
	Let $\g = \gl(n)$, 
	$$\mu = \mathrm{diag}(\underbrace{\lambda_{1}, \dots, \lambda_{1}}_{k_{1}},\underbrace{\lambda_{2}, \dots, \lambda_{2}}_{k_{2}},  \dots, \underbrace{\lambda_{s}, \dots \lambda_{s}}_{k_{s}}).$$
	Then $\g_{\mu} \simeq \gl(k_{1}) \oplus  \dots \oplus \gl(k_{s})$, so $\dim \g_{\mu} - \dim\, [\g_{\mu}, \g_{\mu}] = s$.
	On the other hand, the set of orbits close to $O(\mu)$ which have the same dimension as $O(\mu)$ is parameterized by  the eigenvalues $\lambda_{1}, \dots, \lambda_{s}$, so $\dim_{\mu} \g^{*}(\mu) - \dim O(\mu)$  is also equal to $s$, and inequality (\ref{ineq2}) turns into equality. 
\end{Example}

\paragraph{What happens if the transverse Poisson structure is linearizable} Let $M$ be a Poisson manifold, and $\mu \in M$. Recall that, by the Weinstein splitting theorem \cite{Weinstein1}, $M$ can be locally decomposed into the direct product of a symplectic manifold and a manifold with a Poisson structure vanishing at $\mu$. This latter Poisson structure is unique up to a diffeomorphism and is called the transverse Poisson structure at the point $\mu$. More details can be found in \cite{Zung}.\newpage
	In the case when $M$ is the dual $\g^{*}$ of a Lie algebra $\g$, the linear part of the transverse Poisson structure at a point $\mu$ is the Lie-Poisson structure of the stabilizer $\g_{\mu}$. Consequently, if the transverse Poisson structure at a point $\mu$ is linearizable, then the Lie-Poisson structure on $\g^*$ can be locally decomposed into the direct product of a symplectic structure and the Lie-Poisson structure of the stabilizer $\g_{\mu}$, which allows to prove the following.
\begin{Proposition}\label{lin}
	Assume that the transverse Poisson structure at a point $\mu$ is linearizable. Then $\g^{*}(\mu)$ is smooth at the point $\mu$, and $$\dim_{\mu} \g^{*}(\mu) - \dim O(\mu) = \dim \g_{\mu} - \dim [\g_{\mu}, \g_{\mu}].$$
\end{Proposition}
\begin{Example}[M. Duflo, see \cite{Weinstein, Zung}]
	Let $\g$ be a Lie algebra given by the following linear Poisson structure
	$$
		\ddx{x_{1}} \wedge (x_{2}\ddx{x_{2}} + x_{3}\ddx{x_{3}} + 2x_{4}\ddx{x_{4}}) + x_{4}\ddx{x_{2}}\wedge\ddx{x_{3}}.
	$$
	Consider $\mu \in \g^*$ with $x_{4} = 0$ and $x_{2}^{2}+x_{3}^{2} > 0$. Then the stabilizer of $\mu$ is Abelian: $\dim\, [\g_{\mu}, \g_{\mu}] = 0$. \\On the other hand, $\mathrm{codim}\, \g^{*}(\mu) = 1$. Consequently, the transverse Poisson structure at $\mu$ is not linearizable.
\end{Example}

\section{Semisimple case: the derived algebra of a centralizer and sheets}
\paragraph{Sheets}
In the semisimple case, the coadjoint and the adjoint actions can be identified by the means of the Killing form. This identification maps the variety $\g^{*}_{k}$ to the variety $$\g^{(k)} = \{a \in \g \mid \dim \g^a = k\}$$ where $\g^a = \{x \in \g \mid [a, x] = 0\}$ is the centralizer of $a$.\par
Irreducible components of the varieties $\g^{(k)}$ are called sheets of $\g$.
Recall some facts about the topology of sheets.
\begin{enumerate}
	\item Sheets are not necessarily smooth. However, if $\g$ is a \textit{classical} simple Lie algebra, then sheets are smooth (Im Hof \cite{ImHof}).
	\item Sheets are not necessarily disjoint even in the classical case. However, they are for $\g = \sL(n, \Complex)$ (Kraft and Luna \cite{Kraft}, Peterson \cite{Peterson}).
\end{enumerate}
Study the relation between sheets and the derived algebra of a centralizer.

\paragraph{The main statement in the semisimple case}
 Using Corollary \ref{fundam}, obtain the following.
\begin{Proposition}\label{main}
	Let $\g$ be a real or complex semisimple Lie algebra. Suppose that $a \in \g$ belongs to a sheet $S$, and $S$ is smooth at the point $a$. Then
 the derived algebra of the centralizer of $a$ is orthogonal to $S$ at the point $a$:
		$$
			\langle\, [\g^{a}, \g^{a}], \mathrm{T}_{a}S \,\rangle = 0.
		$$
\end{Proposition}
\newpage
\begin{Corollary}
		 	Let $\g$ be a real or complex semisimple Lie algebra. Suppose that $a \in \g$ belongs to a sheet $S$, and $S$ is smooth at the point $a$. Then
the following three statements are equivalent.
		\begin{enumerate}
		\item The derived algebra of the centralizer of $a$ is exactly the orthogonal complement to $S$ at the point $a$:
		\begin{align}\label{sheetEq}
			[\g^{a}, \g^{a}]  =  (\mathrm{T}_{\mu}S)^{\bot}.
		\end{align}
		\item The dimension of the derived algebra of the centralizer of $\mu$ is equal to the codimension of $S$:
		$$\dim\,  [\g^{a}, \g^{a}] = \mathrm{codim}\, S.$$
		\item The codimension of $[\g^{a}, \g^{a}] $ in $\g^{a}$ is equal to the codimension of $O(a)$ in $S$:
		$$\dim\,\g^{a} - \dim\,  [\g^{a}, \g^{a}] = \dim\, S - \dim O(a).$$	
		\end{enumerate}
		\end{Corollary}

If one of these three conditions is satisfied, 
the picture is the following: the centralizer is the orthogonal complement to the orbit while its derived algebra it is the orthogonal complement to the sheet (Figure \ref{picture}).
\begin{figure}
\psset{unit=0.7cm}
{\begin{pspicture}(-5, -0)(10,8.5)
\pscurve(0,2)(1.2,4)(4,6)
\psline(0,2)(4,0)
\psline(4,6)(8,4)
\pscurve(4,0)(5.2,2)(8,4)
\psline(6,5)(6,8)
\pscurve(2,1)(3.2,3)(6,5)
\psline(7,3)(10,3)(10,8)(2,8)(2,5)
  \psset{linestyle=none}
 \pstextpath[c]{\pscurve(2.2,0.8)(3.4,2.8)(6.2,4.8)}{\\Orbit of $a$}
\rput(5.8,1.5){Sheet}
\rput(7,6.5){$ [\g^{a}, \g^{a}]$}
\rput(10.5,6){$\g^{a}$}
\end{pspicture}}
\begin{center}
\pic{picture}
\end{center}
\end{figure}
\paragraph{The case of a semisimple element}
\begin{Proposition}
	If $\g$ is a real or complex semisimple Lie algebra, and $a \in \g$ is semisimple, then there is only one sheet $S$ passing through $a$, $S$ is smooth at $a$, and equality (\ref{sheetEq}) holds.
\end{Proposition}
\begin{proof}
Since the transverse Poisson structure at a semisimple point is linearizable \cite{Molino}, the proof follows from Proposition \ref{lin}.
\end{proof}
\begin{Corollary}
	Let $\g$ be a compact real Lie algebra. Then all sheets of $\g$ are smooth, disjoint, and the equality (\ref{sheetEq}) holds for every $a \in \g$.
\end{Corollary}
\begin{proof}
	The proof follows from the fact that all elements of a compact algebra are semisimple.
\end{proof}
\paragraph{The case of $\sL(n, \Complex)$}
\begin{Proposition}
	The equality (\ref{sheetEq}) holds for every $a \in \sL(n, \Complex)$.
\end{Proposition}
\begin{proof}
	Denote by $h(\lambda)$ the size of the largest Jordan block of $a$ with the eigenvalue $\lambda$. Prove that
	$$\dim\,\g^{a} - \dim\,  [\g^{a}, \g^{a}] = \dim\, S - \dim O(a) = \left(\sum h(\lambda)\right) - 1.$$
	\begin{enumerate}
		\item $\dim\,\g^{a} - \dim\,  [\g^{a}, \g^{a}]  =  \left(\sum h(\lambda)\right) - 1$. \par
		It suffices to prove this equality for the case when $a$ is nilpotent. This can be easily done by studying the commutation relations for $\g^{a}$ found by O.Yakimova \cite{Yakimova}.
	\item $\dim\, S - \dim O(a) = \left(\sum h(\lambda)\right) - 1$.\par
		This fact is known in the case when $a$ is nilpotent (A.Moreau \cite{Moreau}). The idea of the proof for an arbitrary element is as follows.
		For each eigenvalue $\lambda$, take a sequence of complex numbers $\eps_{1}(\lambda), \dots, \eps_{h(\lambda)}(\lambda)$. 	
		To each Jordan block of $a$ with the eigenvalue $\lambda$, add a diagonal matrix $\mathrm{diag}(\eps_{1}(\lambda), \dots, \eps_{k}(\lambda))$ where $k$ is the size of the block. This gives a family $a_\eps \in \sL(n, \Complex)$ of dimension $\left(\sum h(\lambda)\right) - 1$. It is easy to check that the dimension of the centralizer of each $x \in a_{\eps}$ is equal to the dimension of $\g^a$, so $a_\eps \subset S$ where $S$ is the sheet passing through $a$. At the same time, the family $a_{\eps}$ is transversal to the orbit $O(a)$, so
		$$
			\dim S \geq \dim O(a) + \left(\sum h(\lambda)\right) - 1.
		$$
		On the other hand, by Proposition \ref{main},
		$$
			\dim S - \dim O(a)  \leq \dim\,\g^{a} - \dim\,  [\g^{a}, \g^{a}]  =  \left(\sum h(\lambda)\right) - 1,
		$$		
		so $\dim S - \dim O(a)   =  \left(\sum h(\lambda)\right) - 1$, q.e.d.
		\end{enumerate}
\end{proof}

\paragraph{The case of an arbitrary classical simple Lie algebra}
In the previous version of the present paper \cite{prev}, the following conjecture was formulated.
\begin{Conjecture}
	Let $\g$ be a complex classical simple Lie algebra.
	\begin{enumerate}
		\item If there is only one sheet $S$ passing through $a \in \g$, then equality (\ref{sheetEq}) holds, i.e.
		$$
			[\g^{a}, \g^{a}]  =  \left(\mathrm{T}_{a}S\right)^{\bot}.
		$$
		\item If $a \in \g$ belongs to several sheets $S_{1}, \dots S_{k}$, then
		$$
			[\g^{a}, \g^{a}]  =  \left(\sum_{i=1}^k \mathrm{T}_{a}S_{i}\right)^{\bot}.
		$$		
	\end{enumerate}
\end{Conjecture}
Recently, A.Premet and L.Topley \cite{Premet} have proved the first part of this conjecture for any algebraically closed ground field of characteristic $0$. They have also provided a combinatorial description of those elements $a \in \g$ for which there is only one sheet passing through $a$.\par
The second part of the conjecture remains open.\par

\par
Also note that the conjecture is false for the exceptional Lie algebra $\mathrm{G}_{2}$, as it follows from Remark 3 of \cite{Yakimova}.


\begin{thebibliography}{99}
 \bibitem{Bolsinov} Bolsinov,\,A.V.,  \textit{Compatible Poisson brackets on Lie algebras and completeness of families of functions in involution}, Mathematics of the USSR-Izvestiya, {\bf 38}:1 (1992), 69-90.
 
 \bibitem{JK} Bolsinov,\,A., and Zhang\,P., \textit{Jordan-Kronecker invariants of finite-dimensional Lie algebras}, arXiv:1211.0579 (2012).

  \bibitem{Zung} Dufour,\,J.-P., and Nguyen\,T.Z., ``Poisson Structures and Their Normal Forms'', Birkhauser Basel, 2005.  

 
   \bibitem{ImHof} Im Hof,\,A., \textit{The sheets of a classical Lie algebra}, PhD Thesis, University of Basel, Faculty of Science (2005).
   
   \bibitem{prev} Izosimov,\,A., \textit{The derived algebra of a stabilizer, families of coadjoint orbits, and sheets}, arXiv:1202.1135v2 (2012).
   
   \bibitem{Kirillov} Kirillov,\,A.A., \textit{Repr\'esentations unitaires des groupes de Lie nilpotents}, Uspehi Mat. Nauk, \textbf{17} (1962), 57-110.

   
       \bibitem{Kraft} Kraft,\,H., \textit{Parametrisierung von Konjugationsklassen in $\sL(n)$}, Math. Ann., \textbf{234} (1978), 209-220.
       
    \bibitem{Molino} Molino,\,P., \textit{Structure transverse aux orbites de la repr\'esentation coadjointe: le cas des orbites r\'eductives}, S\'eminaire Gaston Darboux de G\'eom\'etrie Diff\'erentielle \`a Montpellier (1983-1984), 55-62.

 \bibitem{Moreau} Moreau,\,A., \textit{On the dimension of the sheets of a reductive Lie algebra}, Journal of Lie Theory, \textbf{18}:3 (2008), 671-696.
   
  \bibitem{Vinberg} Panyushev,\,D., \textit{The index of a Lie algebra, the centraliser of a nilpotent element, and the normaliser of the centraliser}, Math. Proc. Camb. Phil. Soc. \textbf{134} (2003), 41-59.
    
     \bibitem{Peterson} Peterson,\,D., \textit{Geometry of the adjoint representation of a complex semisimple Lie algebra}, Harvard thesis (1978).
    
    \bibitem{Premet} Premet,\,A., and Topley,\,L., \textit{Derived subalgebras of centralisers and finite W-algebras}, arXiv:1301.4653 (2013).
    
 \bibitem{Weinstein} Weinstein,\,A., \textit{Poisson structures and Lie algebras}, Ast\'erisque, hors s\'erie (1985), 421-434.
 
  \bibitem{Weinstein1} Weinstein,\,A., \textit{The local structure of Poisson manifolds}, J. Diff. Geom. \textbf{18} (1983), 523-557.
 
  \bibitem{Yakimova} Yakimova,\,O., \textit{On the derived algebra of a centraliser}, Bull. des Sciences Math. \textbf{134}:6 (2010), 579-587; arxiv:1003.0602v1.
 \end{thebibliography}
\end{document}